\documentclass[a4paper]{amsart}
\usepackage{amsmath,amsthm,amssymb}
\usepackage{amscd}
\usepackage[dvipdfmx,hiresbb]{graphicx}
\usepackage{mathrsfs}
\usepackage[dvipdfmx, usenames]{color}
\usepackage{graphicx}

\usepackage{tikz}
\usetikzlibrary{positioning, intersections, calc, arrows.meta,math}
\usetikzlibrary{patterns}
\usetikzlibrary{decorations.pathreplacing}

\usepackage{multicol}
\usepackage[arrow,matrix]{xy}

\usepackage{multirow}
\usepackage[normalem]{ulem}
\useunder{\uline}{\ul}{}
\usepackage {diagbox}
\usepackage{ulem}
\usepackage{float}
\newtheorem{thm}{Theorem}[section]
\newtheorem{lem}[thm]{Lemma}
\newtheorem{cor}[thm]{Corollary}
\newtheorem{prop}[thm]{Proposition}

\theoremstyle{definition}

\newtheorem{defn}[thm]{Definition}
\newtheorem{remark}[thm]{Remark}

\theoremstyle{remark}

\newcommand{\calG}{\mathcal{G}}

\newcommand{\calI}{\mathcal{I}}

\newcommand{\calN}{\mathcal{N}}

\newcommand{\bbR}{\mathbb{R}}

\newcommand{\bbZ}{\mathbb{Z}}

\newcommand{\frako}{\mathfrak{o}}

\newcommand{\ora}{\overrightarrow}

\numberwithin{equation}{subsection}

\def\XXint#1#2#3{{\setbox0=\hbox{$#1{#2#3}{\int}$}
\vcenter{\hbox{$#2#3$}}\kern-.5\wd0}}

\usepackage[abbrev,nobysame]{amsrefs}

\makeatletter
\def\@makefnmark{%
\leavevmode
\raise.9ex\hbox{\check@mathfonts
\fontsize\sf@size\z@\normalfont%
\@thefnmark}%
}
\makeatother

\title{Weighted sums of rooted spanning forests on cycles with pendant edges}
\author[H.~Fujita, K.~Hasegawa, Y.~Inaba, T.~Kondo]{Hajime Fujita, Kimiko Hasegawa,Yukie Inaba, Takefumi Kondo}
\address[H.~Fujita]{Japan Women's university}
\email{fujitah@fc.jwu.ac.jp}
\address[K.~Hasegawa]{Japan Women's university}
\email{m1716072hk@ug.jwu.ac.jp}
\address[Y.~Inaba]{JAL Information Technology Co., Ltd.}
\address[T.~Kondo]{Kagoshima university}
\email{takefumi@sci.kagoshima-u.ac.jp}
\begin{document}
\maketitle
\begin{abstract}
We derive two formulas for the weighted sums of rooted spanning forests of particular sequence of graphs by using the matrix tree theorem. 
We consider cycle graphs with edges so called the pendant edges. 
One of our formula can be described as a variable transformation of the Chebyshev polynomial. 
They have particular algebraic properties. 
\end{abstract}


\section{Introduction}
Counting problem of subgraphs with specific conditions is a fundamental problem in combinatorial graph theory. 
A typical example is the counting of {\it spanning forests}, and the solution is given by the famous {\it matrix tree theorem}. 
For a weighted graph the solution of the counting problem gives a polynomial as a generating function. 
In this paper we consider the counting problem of {\it rooted spanning forests} of a given graph with a fixed node set of the vertex set. 
In this setting the polynomial is called the {\it weighted sum} of rooted spanning forests. 

For a finite graph with boundary, a spanning forest in which every component tree contains some boundary points
is called a {\it grove}, and furthermore, a grove whose component tree contains exactly one boundary point 
is called an {\it uncrossing} \cite{Kenyon}. 
In this paper we treat a finite graph with boundary as a finite graph with a decomposition of the vertex set into nodes and internal vertices. 
In our terminology the uncrossing is nothing other than the rooted spanning forest.  
Enumeration of such objects appear in a combinatorial description of the response matrix, or Dirichlet-to-Neumann matrix
of networks, and played important roles in the study of the discrete EIT (Electrical Impedance Tomography) problem.

For any positive integer $n$ we compute the weighted sum of a cycle graph attaching specific $n$ edges so called {\it $n$ pendant edges} (e.g., see \cite{Barrientos}). We consider the weight with rotational invariance. 
As main theorems we derive formulas for the non-oriented version and the oriented version. 
The non-oriented version gives a polynomial obtained as a variable transformation of the Chebyshev polynomial.
The oriented version gives a polynomial described by the binomial coefficients. 
It turns out that both of them have factorization formulas described by minimal polynomials of $\zeta_n+\zeta^{-1}_n$ and $\zeta_n$ for $\zeta_n=e^{\frac{2\pi \sqrt{-1}}{n}}$, and they 
produce families of polynomials characterized by specific properties related with compositions and divisibilities. 
Several algebraic properties including the factorization formula and the characterization will be discussed in the forthcoming paper \cite{HasegawaSugiyama}. 

This paper is organized as follows. 
In Section~\ref{section:Preliminaries} we fix several notations and terminologies of graph theory. 
We also refer our main tool, the matrix tree theorem (Theorem~\ref{thm:Matrix-Tree}, Theorem~\ref{thm:Matrix-Tree2}). 
In Section~\ref{section:Main theorems} we give main theorems (Theorem~\ref{thm:main}, Theorem\ref{thm:main2}), computations of the weighted sum of the cycle graph with $n$ pendant edges. 
In Section~\ref{section:Corollaries} we exhibit the explicit computations, which posses particular algebraic properties. 
We also give corollaries on the concavity and unimodality of the coefficients of our weighted sums (Corollaries~\ref{cor:x(x+4)}, Corollaries~\ref{cor:logconc}). 
We give statements of algebraic properties of our weighted sums following \cite{HasegawaSugiyama} (Theorem~\ref{thm:Factorization}, Theorem~\ref{thm:Characterization}). 

\medskip

\noindent
{\bf Acknowledgement.} 
The authors are grateful to Rin Sugiyama for fruitful discussions on geometric and algebraic properties of our weighted sums. 
\section{Preliminaries}\label{section:Preliminaries}

\subsection{Notations and terminologies}
In this article we consider finite simple graphs, i.e., graphs having finite vertex set and edge set without self loops or multiple edges. 
For a graph $\calG$ let $V(\calG)$ be the vertex set and $E(\calG)$ the edge set of $\calG$. 
We denote the edge connecting vertices $v$ and $v'$ by $vv'\in E(\calG)$.  
We may fix an arbitrarily ordering of $V(\calG)$ if necessarily.  
A {\it weight} of a given graph $\calG$ is a function 
\[
w:E(\calG)\to \bbR.  
\]The value $w(e)$ for an edge $e\in E(\calG)$ is called the {\it weight of $e$}. 
We often regard $w(e)$ as an indeterminant attached on $e\in E(\calG)$. 
A weighted graph is a pair $(\calG,w)$ consisting of a graph $\calG$ and its weight $w$.  
For any graph $\calG$ we fix a decomposition of $V(\calG)$ into two disjoint subsets $\calN(\calG)$ and $\calI(\calG)$, 
which are called the {\it set of nodes} and the {\it set of internal vertices} respectively.  

\begin{defn}
A {\it rooted spanning forest} of a graph $\calG$ is a subgraph $\calG'$ of $\calG$ which satisfies the following conditions. 
\begin{itemize}
\item $V(\calG')=V(\calG)$. 
\item $\calG'$ has no cycles. 
\item Each connected component of $\calG'$ has exactly one vertex in $\calN(\calG)$. 
\end{itemize}
We denote the set of all rooted spanning forests of $\calG$ by ${\rm RSF}(\calG)$. 
\end{defn}

\begin{defn}
Let $\calG=(\calG,w)$ be a weighted graph. 
The {\it weighted sum of rooted spanning forests of $\calG$} is a partition function defined by 
\[
Z(\calG)=\sum_{\calG'\in {\rm RSF}(\calG)}\prod_{e\in E(\calG')}w(e). 
\] 
\end{defn}

For a finite set $A$ let the symbol $|A|$ denotes the cardinality of $A$. 

\subsection{Laplacian matrix and matrix tree theorem}
In this subsection let $\calG=(\calG, w)$ be a weighted graph with a decomposition $V(\calG)=\calN(\calG)\sqcup\calI(\calG)$. 
\subsubsection{Non-oriented version}
\begin{defn}[The Laplacian matrix]
Let $L(\calG)=(L(\calG)_{vv'})$ be the matrix whose entries are indexed by $V\times V$ and defined as 
\[
L(\calG)_{vv'}=
\begin{cases}
-w(vv') \quad (v\neq v' \ {\rm and} \ vv'\in E(\calG)) \\ 
 \quad 0  \quad (v\neq v' \ {\rm and} \ vv'\notin E(\calG)) \\ 
\displaystyle\sum_{vv''\in E(\calG)}w(vv'') \quad (v=v'). 
\end{cases}
\]
\end{defn}

\begin{thm}[The matrix tree theorem,\cite{Wagner}]\label{thm:Matrix-Tree}
If the node set consists of a single vertex, then we have 
\[
Z(\calG)=(-1)^{i+j}\det(L_{(v_iv_j)})(\calG)), 
\]where $L_{(v_iv_j)}(\calG)$ is  the $(|V(\calG)|-1)\times(|V(\calG)|-1)$ matrix obtained by deleting a row vector  corresponding to $i$-th vertex $v_i$ and a column vector corresponding to $j$-th vertex $v_j$ from the Laplacian matrix $L(\calG)$ of $\calG$. 
\end{thm}

\subsubsection{Oriented version}
We fix an orientation ${\mathfrak o}$ of the edge set $E(\calG)$ of $\calG$. 
Let $\calG^\frako=(V(\calG), E^\frako(\calG))$ be the associated oriented graph.  
We denote the oriented edge from $v$ to $v'$ by $\overrightarrow{vv'}\in E^\frako(\calG)$. 

\begin{defn}[The oriented Laplacian matrix]
Let $L^{\rm ori}(\calG^\frako)=(L^{\rm ori}(\calG^\frako)_{vv'})$ be the matrix whose entries are indexed by $V\times V$ and defined as 
\[
L^{\rm ori}(\calG^\frako)_{vv'}=
\begin{cases}
-w(\overrightarrow{vv'}) \quad (v\neq v'  \ {\rm and} \ \overrightarrow{vv'}\in E^\frako(\calG)) \\ 
\quad 0 \quad  (v\neq v'  \ {\rm and} \ \ora{vv'}\notin E^\frako(\calG)) \\ 
\displaystyle\sum_{\overrightarrow{v'' v} \in E(\mathcal{G}^\mathfrak{o})} w(\overrightarrow{v'' v}) \quad (v=v'). 
\end{cases}
\]
\end{defn}


\begin{defn}[The oriented rooted spanning forest]
An {\it oriented rooted spanning forest} of the oriented graph $\calG^\frako$ is a rooted spanning forest of $\calG$ whose connected component has the outward orientation starting from the node. 
Let ${\rm RSF^{ori}}(\calG^\frako)$ be the set of all oriented rooted spanning forests of $\calG^\frako$. 
\end{defn}

\begin{defn}[The weighted sum of oriented rooted spanning forest]
The {\it weighted sum of oriented rooted spanning forests of $\calG^\frako$} is a partition function defined by 
\[
Z^{\rm ori}(\calG^\frako)=\sum_{\calG'\in {\rm RSF^{ori}}(\calG^\frako)}\prod_{\vec e\in E(\calG')}w(\vec e). 
\] 
\end{defn}

\begin{thm}[The oriented matrix tree theorem,\cite{Takasaki, Tutte}]\label{thm:Matrix-Tree2}
If the node set consists of a single vertex $\calN(\calG)=\{v_j\}$, then we have 
\[
Z^{\rm ori}(\calG^\frako)=(-1)^{i+j}\det(L^{\rm ori}_{(v_iv_j)}(\calG^\frako)), 
\]
where $L^{\rm ori}_{(v_iv_j)}(\calG^\frako)$ is  the $(|V(\calG)|-1)\times(|V(\calG)|-1)$ matrix obtained by deleting a row vector  corresponding to $i$-th vertex $v_i$ and the column vector corresponding to the node $v_j$ from the Laplacian matrix $L^{\rm ori}(\calG^\frako)$ of $\calG^\frako$. 
\end{thm}

\section{Main theorems : the weighted sum of cycles with $n$-pendant edges}\label{section:Main theorems}
Let $n$ be a positive integer and $C_n$ the cycle graph with $n$ vertices
\[
V(C_n)=\{v_1, \ldots, v_n\}
\] and $n$ edges  
\[
E(C_n)=\{e_1=v_1v_2, e_2=v_2v_3, \ldots, e_{n-1}=v_{n-1}v_n, e_n=v_nv_1\}. 
\]
\begin{defn}
Let $V'_n=\{v_1',\ldots, v_n'\}$ be a copy of $V(C_n)$. 
We define the {\it cycle with $n$ pendant edges} $\calG_n$ by the following vertex set and the edge set. 
\[
V(\calG_n)=V(C_n)\cup V'_n, 
\]
\[
E(\calG_n)=E(C_n)\cup\{e_i' =v_iv_i'\ | \ i=1,\ldots, n\}. 
\]
We define the nodes of $\calG_n$  by $\calN(\calG_n)=V'_n$.  
In Figure~\ref{fig:1} and Figure~\ref{fig:2}  a node is depicted as a black circle and an internal vertex is depicted as white circle.

Unless otherwise stated we use a weight $w$ of the cycle with $n$ pendant edges of the form 
\[
w(e_i')=a, \quad w(e_i)=b \quad (i=1,\ldots, n)
\]for two indeterminant $a$ and $b$. 
\end{defn}

\begin{figure}[H]
\includegraphics[scale=0.8]{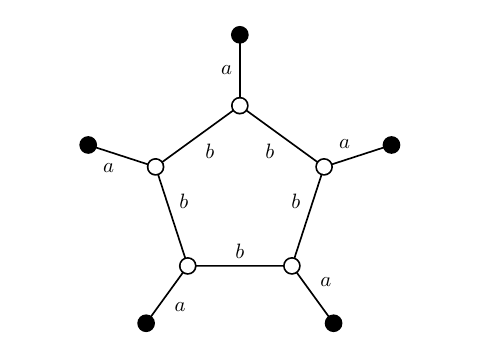}
\caption{$\calG_5$}\label{fig:1}
\end{figure}


\begin{figure}[H]
\includegraphics[scale=0.8]{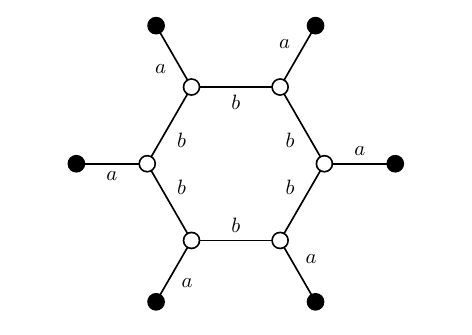}
\caption{$\calG_6$}\label{fig:2}
\end{figure}

\begin{remark}
A vertex of degree 1 is often called a {\it pendant vertex},   
and the cycle with $n$-pendant edges is also called the {\it sun let graph with $n$ pendant vertices}. 
See \cite{Barrientos} or \cite{Vernold} for example. 
\end{remark}

\subsection{The non-oriented case}
In our setting the weighted sum of $\calG_n$ is a two variable polynomial in $a$ and $b$ : 
\[
F_n(a,b):=Z(\calG_n). 
\]

It can be seen that $F_n(a,b)$ is a homogeneous polynomial of degree $n$ with $F_n(0,b)=0$. 
Our main result is the explicit description of $F_n(a,b)$ using the {\it Chebyshev polynomial of the first kind}. 
\begin{defn}
For each positive integer $n$ the {\it Chebyshev polynomial of the first kind} is a polynomial $T_n(x)$ satisfying the condition   
\[
T_n(\cos\theta)=\cos(n\theta)
\] for all $\theta\in\bbR$. 
\end{defn}

\begin{prop}\label{prop:characterization of T_n}
The Chebyshev polynomial $T_n(x)$ exists uniquely and can be computed by the recursive formula 
\[
\begin{cases}
T_1(x)=x \\ 
T_2(x)=2x^2-1 \\ 
T_{n+2}(x)=2xT_{n+1}(x)-T_n(x) \quad (n=1,2,\ldots,). 
\end{cases}
\]
\end{prop}

The following is the main theorem. 

\begin{thm}\label{thm:main}
$
F_n(a,b)=2b^n\left(T_n\left(\frac{a}{2b}+1\right)-1\right)
$
\end{thm}


To show Theorem~\ref{thm:main} we introduce two lemmas below. 

\begin{lem}\label{lem:identifying nodes}
Let $\calG=(\calG, w)$ be a weighted graph with a decomposition $V(\calG)=\calN(\calG)\sqcup\calI(\calG)$. 
Let $L_\calI(\calG)$ be the $(|V(\calG)|-|\calN(\calG)|)\times(|V(\calG)|-|\calN(\calG)|)$ submatrix of the 
Laplacian matrix $L(\calG)$ corresponding to the internal vertices $\calI(\calG)$. 
We have 
\[
\det(L_\calI(\calG))=Z(\calG). 
\]
\end{lem}
\begin{proof}
Let $\calG'$ be the graph obtained from $\calG$ by identifying all nodes. 
Note that $\calG'$ has only one node $v_0$. One can see that there is a natural bijective between ${\rm RSF}(\calG)$ and ${\rm RSF}(\calG')$. 
In particular we have $Z(\calG)=Z(\calG')$. 
The Laplacian $L(\calG')$ is a $(|\calI(\calG)|+1)\times(|\calI(\calG)|+1)$ matrix of  the form 
\[
L(\calG')=
\begin{pmatrix}
* & * \\
* & L_\calI(\calG)
\end{pmatrix}. 
\]By taking $L_{(v_0v_0)}(\calG')=L_\calI(\calG)$ we have 
\[
Z(\calG)=Z(\calG')=\det(L_{(v_0v_0)}(\calG'))=\det(L_\calI(\calG))
\]by Theorem~\ref{thm:Matrix-Tree}. 
\end{proof}

\begin{remark}
Lemma~\ref{lem:identifying nodes} is refered in \cite[Theorem(Excesice)~6.1]{Wagner}. 
This technique \lq\lq identifying  all nodes\rq\rq \ in the proof of Lemma~\ref{lem:identifying nodes} is also used  in \cite[Lemma~4.3.1]{KenyonWilson1}. 
\end{remark}

\begin{lem}\label{lem:n-gon}
For a cycle with $n$ pendant edges $\calG_n$ we have 
\[
\det(L_\calI(\calG_n))=2b^n\left(T_n\left(\frac{a}{2b}+1\right)-1\right). 
\]
\end{lem}
\begin{proof}
We first see that 
\[
L_\calI(\calG_n)=
\begin{pmatrix}
 a+2b & -b & 0 & \cdots & 0 & -b \\
 -b &  a+2b & -b & \cdots & 0 & 0 \\
 0 & -b & a+2b & \cdots & 0 & 0 \\ 
 \vdots  & \vdots & \vdots & \vdots & \vdots & \vdots \\
 0 & 0 & 0 & \cdots & a+2b & -b \\ 
  -b & 0 & 0 & \cdots & -b & a+2b 
\end{pmatrix}
=bL(C_n)+aE, 
\]where $C_n$ is the cycle graph with $\calN(C_n)=\emptyset$ and equipped with the constant weight $w\equiv 1$. 
Then we have 
\[
\det(L_\calI(\calG_n))=\det(bL(C_n)+aE)=b^n\det\left(L(C_n)+\frac{a}{b}E\right). 
\]
As it is stated in \cite[Section~1.4.3]{BrouwerHaemers}, one can see that eigenvalues of $L(C_n)$ are 
\[
2-2\cos\left(\frac{2\pi k}{n}\right) \quad (k=0,1,\ldots, n-1). 
\] 
On the other hand the definition $T_n(x)=T_n(\cos\theta)=\cos(n\theta)$ shows that the roots of $T_n(x)-1=0$ are
\[
\cos\left(\frac{2\pi k}{n}\right) \quad (k=0,1,\ldots, n-1), 
\]and hence, two polynomials $\det\left(L(C_n)+xE\right)$ and $T_n\left(\frac{x}{2}+1\right)-1$ have the same roots. 
Since the coefficient of $x^n$ in $T_n(x)$ is $2^{n-1}$ one has 
\[
T_n\left(\frac{x}{2}+1\right)-1=2^{n-1}\left(\frac{x}{2}+1\right)^n+\cdots =
\frac{1}{2}x^n+\cdots. 
\]It implies that $\det\left(L(C_n)+xE\right)=2\left(T_n\left(\frac{x}{2}+1\right)-1\right)$, and hence, we have 
\[
\det(L_\calI(\calG_n))=b^n\det\left(\Delta(C_n)+\frac{a}{b}E\right)=
2b^n\left(T_n\left(\frac{a}{2b}+1\right)-1\right). 
\]
\end{proof}

\begin{proof}[Proof of Theorem~\ref{thm:main}]
By using Lemma~\ref{lem:identifying nodes} and Lemma~\ref{lem:n-gon} we have 
\[
F_n(a,b)=Z(\calG_n)=\det(L_\calI(\calG_n))=2b^n\left(T_n\left(\frac{a}{2b}+1\right)-1\right). 
\]
\end{proof}

\subsection{The oriented case}
For each $n$ we fix an orientation $\frako$ of $\calG_n$ as follows. 
We fix the counter clockwise orientation of edges in $C_n$ and the orientation of pendant edges in such a way that $\ora{v_i'v_i}\in E^\frako$.  See Figure~\ref{fig:3} and Figure~\ref{fig:4}

\begin{figure}[H]
\includegraphics[scale=0.8]{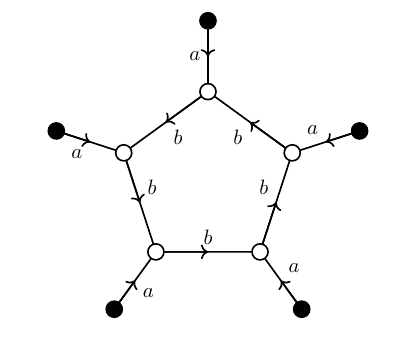}
\caption{$\tilde\calG_5$}\label{fig:3}
\end{figure}

\begin{figure}[H]
\includegraphics[scale=0.8]{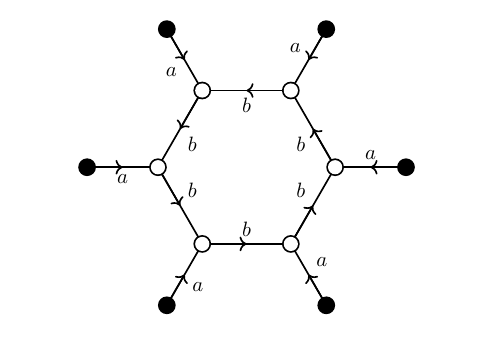}
\caption{$\tilde\calG_6$}\label{fig:4}
\end{figure}

Let $\tilde\calG_n$ be the weighted oriented graph $\tilde\calG_n=(\calG^\frako_n, w)$. 
We put  
\[
\tilde F_n(a,b):=Z^{\rm ori}(\tilde\calG_n). 
\]

\begin{thm}\label{thm:main2}
$\tilde F_n(a,b)=(a+b)^n-b^n$. 
\end{thm}
\begin{proof}
Consider the oriented weighted graph $\tilde\calG_n'$ by identifying all nodes in $\tilde\calG_n$. 
We see that 
\[
L^{\rm ori}(\tilde\calG_n')
= \begin{pmatrix}
0 & -a & -a & -a & \cdots & -a\\
0 & a+b & -b & 0 & \cdots & 0 \\
0 & 0 & a+b & -b & \cdots & 0 \\
\cdots & \cdots & \cdots & \cdots & \cdots & \cdots \\
 0 & -b & 0 & \cdots & 0 & a+b
\end{pmatrix}, 
\]
and we use the oriented version of Lemma~\ref{lem:identifying nodes} so that we have
\[
\tilde F_n(a,b)=Z^{\rm ori}(\tilde\calG_n)=Z^{\rm ori}(\tilde\calG'_n)=\det(L_\calI(\tilde\calG_n)), 
\]where 
\[
L_\calI(\tilde\calG_n)=
\begin{pmatrix}
a+b & -b & 0 & \cdots & 0 \\
0 & a+b & -b & \cdots & 0 \\
\cdots & \cdots & \cdots & \cdots & \cdots \\
-b & 0 & \cdots & 0 & a+b 
\end{pmatrix}. 
\] By the direct computation we have 
\[
\tilde F_n(a,b)=\det(L_\calI(\tilde\calG_n))=(a+b)^n-b^n. 
\]
\end{proof}

\begin{remark}
Figure~\ref{fig:5} is an example of the oriented rooted spanning forests, and Figure~\ref{fig:6} is not.  
Theorem~\ref{thm:main2} can be obtained by a purely combinatorial argument without using the matrix tree theorem. 
 In fact one can count that the number of oriented rooted spanning forests with $k(=0, \ldots, n-1)$ internal edges is equal to the binomial coefficient $\begin{pmatrix} n \\ k \end{pmatrix}$, and hence, the formula  
 \[
\tilde F_n(a,b)=\sum_{k=0}^{n-1}\begin{pmatrix} n \\ k \end{pmatrix}a^{n-k}b^{k}=(a+b)^n-b^n
 \]holds. 
 \end{remark}

\begin{figure}[H]
\includegraphics[scale=0.8]{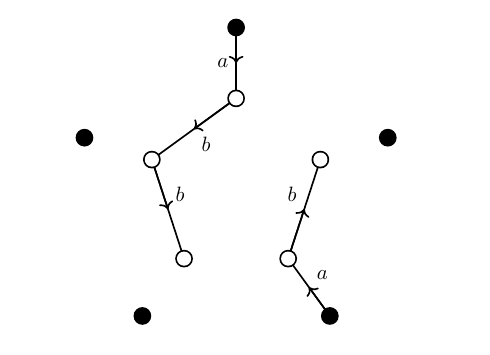}
\caption{an oriented rooted spanning forest}\label{fig:5}
\end{figure}

\begin{figure}[H]
\includegraphics[scale=0.8]{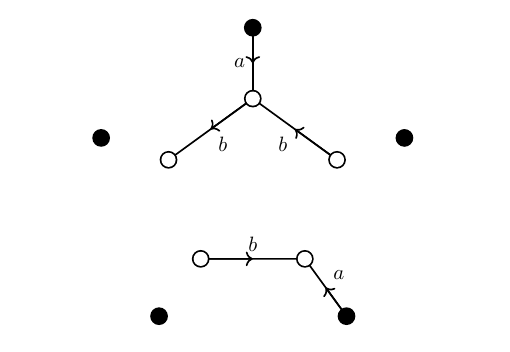}
\caption{not an oriented rooted spanning forest}\label{fig:6}
\end{figure}

\section{Corollaries and further discussions}\label{section:Corollaries}
In this section we specialize $F_n(a,b)$ as 
\[
F_n(x):=F_n(x,1)=2\left(T_n\left(\frac{x}{2}+1\right)-1\right). 
\]One can recover $F_n(a,b)$ from $F_n(x)$ by 
\[
F_n(a,b)=b^nF_n\left(\frac{a}{b}\right). 
\]
Theorem~\ref{thm:main} gives the following corollary. 
\begin{cor}\label{cor:x(x+4)}
The polynomial $c_n(x)$ defined by 
\[
c_n(x):=
\begin{cases}
x \quad (n : {\rm odd}) \\ 
x(x+4) \quad (n : {\rm even}). 
\end{cases}
\]
is a divisor of $F_n(x)$. 
\end{cor}
\begin{proof}
We have 
\[
 F_n(0)=F_n(0,1)=0  
\]for all $n$, 
and if $n$ is even, then by using Theorem~\ref{thm:main} we have 
\[
F_n(-4)=2^n(T_n(-1)-1)=2^n(\cos(n\pi)-1)=0.  
\]
\end{proof}
We also introduce the polynomial 
\[
\tilde F_n(x)=\tilde F_n(x,1)=(x+1)^n-1. 
\]
One can see that $\tilde F_n(x)$ has two real roots $x=0, -2$ when $n$ is even and one real root $x=0$ when $n$ is odd. 
In contrast we have shown in the proof of Lemma~\ref{lem:n-gon} that all roots of $F_n(x)$ are real roots 
\[
\omega_k:=\zeta_n^k+\zeta_n^{-k}-2=2\left(\cos\frac{2\pi k}{n}-1\right) \ (k=0,1,\ldots, n-1). 
\]
 We have the following corollary by \cite[Lemma~1.1]{Brabden}. 
\begin{cor}\label{cor:logconc}
Let $\alpha_j \ (j=1,\ldots,n)$ be the coefficient of $x^j$ in $F_n(x)$. 
The sequence $\{\alpha_j\}_{j=1,\ldots, n}$ is log-concave, i.e., the inequality  
\[
\alpha_j^2\geq \alpha_{j-1}\alpha_{j+1}
\]holds for all $j=2,\ldots, n-1$. In particular they are unimodal, i.e., there exists $j'\in \{1,\ldots,n\}$ such that 
\[
\alpha_1\leq\cdots\leq \alpha_{j'-1}\leq \alpha_{j'}\geq \alpha_{j'+1}\geq \cdots \geq \alpha_n
\]holds.  
\end{cor}

\begin{remark}
The roots $\omega_k$ of $F_n(x)=0$ are the real parts of $2(\zeta_n^k-1)$ for $k=0,\ldots, n-1$, which can be illustrated as shown in Figure~\ref{fig:fig} together with the circle and its inscribed $n$-gon. 

\begin{figure}
\includegraphics[scale=0.5]{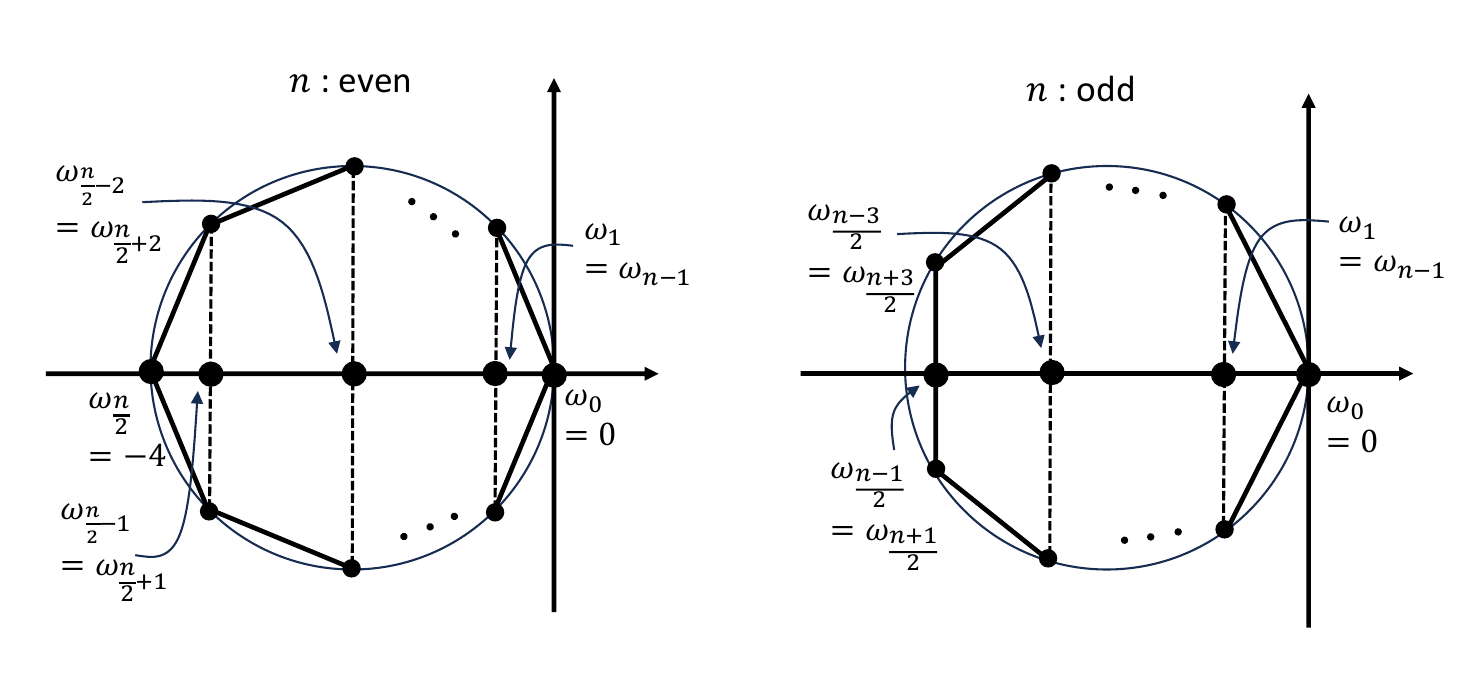}
\caption{Configurations of roots of $F_n(x)$} \label{fig:fig}
\end{figure}
\end{remark}

We have the following explicit computations, and these results suggest more detailed factorization of $F_n(x)$. 
\begin{enumerate}
\item $F_1(x)=x$
\item $F_2(x)=x^2+4x=x(x+4)$
\item $F_3(x)=x^3 + 6x^2 + 9x = x(x + 3)^2$
\item $F_4(x)=x^4 + 8x^3 + 20x^2 + 16x = x(x + 2)^2(x + 4)$
\item $F_5(x)=x^5 + 10x^4 + 35x^3 + 50x^2 + 25x= x(x^2 + 5x + 5)^2$
\item $F_6(x)=x^6 + 12x^5 + 54x^4 + 112x^3 + 105x^2 + 36x= x(x + 1)^2(x + 3)^2(x + 4)$
\item $F_7(x)=x^7 + 14x^6 + 77x^5 + 210x^4 + 294x^3 + 196x^2 + 49x \\ \hspace{0.85cm}= x(x^3 + 7x^2 + 14x + 7)^2$
\item $F_8(x)=x^8 + 16x^7 + 104x^6 + 352x^5 + 660x^4 + 672x^3 + 336x^2 + 64x \\ 
\hspace{0.85cm}= x(x + 2)^2(x + 4)(x^2 + 4x + 2)^2$
\item $F_9(x)=x^9 + 18x^8 + 135x^7 + 546x^6 + 1287x^5 + 1782x^4 + 1386x^3 + 540x^2 + 81x \\
\hspace{0.85cm} = x(x + 3)^2(x^3 + 6x^2 + 9x + 3)^2$
\item $F_{10}(x)=x^{10} + 20x^9 + 170x^8 + 800x^7 + 2275x^6 + 4004x^5 \\ 
\hspace{5cm}+4290x^4 + 2640x^3 + 825x^2 + 100x \\
 \hspace{0.95cm} = x(x + 4)(x^2 + 3x + 1)^2(x^2 + 5x + 5)^2$
\item $F_{11}(x)=x^{11} + 22x^{10} + 209x^9 + 1122x^8 + 3740x^7 + 8008x^6 + 11011x^5 \\ 
\hspace{5cm} +9438x^4 + 4719x^3 + 1210x^2 + 121x \\ 
\hspace{1cm}= x(x^5 + 11x^4 + 44x^3 + 77x^2 + 55x + 11)^2$ 
\item $F_{12}(x)=x^{12} + 24x^{11} + 252x^{10} + 1520x^9 + 5814x^8 + 14688x^7 + 24752x^6 \\ 
\hspace{4.5cm }+27456x^5 + 19305x^4 + 8008x^3 + 1716x^2 + 144x \\ 
\hspace{1cm}= x(x + 1)^2(x + 2)^2(x + 3)^2(x + 4)(x^2 + 4x + 1)^2$
\end{enumerate}

In fact the factorizations of $\tilde F_n(x)$ and $F_n(x)$ in $\bbZ[x]$ into irreducible polynomials are shown in \cite{HasegawaSugiyama}. 
To explain divisors of $\tilde F_n(x)$ and $F_n(x)$ we introduce  auxiliary polynomials as follows. 
Let $\Phi_n(x)$ be the $n$-th cyclotomic polynomial, i.e., the minimal polynomial of $\zeta_n=e^{\frac{2\pi\sqrt{-1}}{n}}$. 
Let $\Psi_n(x)$ be the minimal polynomial of $\zeta_n+\zeta_n^{-1}=2\cos\frac{2\pi }{n}$. 

\begin{thm}[\cite{HasegawaSugiyama}]\label{thm:Factorization}
We have the following factorizations of $F_n(x)$ and $\tilde F_n(x)$ in $\bbZ[x]$. 
\[
\tilde F_n(x)=\prod_{k|n}\Phi_k(x+1), \quad 
F_n(x)=c_n(x)\prod_{k|n, k>2}\left(\Psi_k(x+2)\right)^2. 
\]
\end{thm}

The above explicit computations include several divisibility, such as
\[
F_3(x) | F_{12}(x), \ F_4(x) | F_8(x), \ F_5(x) | F_{10}(x), \cdots . 
\]
By using Theorem~\ref{thm:Factorization}, the following algebraic characterization of $F_n(x)$ and $\tilde F_n(x)$ 
including the divisibilities is shown.
\begin{thm}[\cite{HasegawaSugiyama}]\label{thm:Characterization}
Suppose that a sequence of polynomials $\{P_n(x)\}_{n=1,2,\ldots}$ of integer coefficients satisfies the following conditions for all $n,m=1,2,\ldots$.
\begin{enumerate}
\item $\deg(P_n(x))=n$ 
\item $P_n(x)$ is monic. 
\item $P_n(P_m(x))=P_{nm}(x)=P_m(P_n(x))$
\item $P_n(x) | P_m(x) \Longleftrightarrow n | m$ 
\end{enumerate} 
Then we have $P_n(x)=F_n(x)$ or $\tilde F_n(x)$ for all $n=1,2\ldots$. 
\end{thm}

More detailed algebraic properties will be shown in \cite{HasegawaSugiyama},  
and graph theoretical interpretation of the divisibility will be discussed in the subsequent papers. 

\bibliography{reference}

\end{document}